\documentclass[12pt]{amsart}

\usepackage[latin1]{inputenc}
\usepackage{amsmath,amsthm,amssymb,amscd}

\usepackage[colorlinks=true, pdfstartview=FitV, linkcolor=blue, citecolor=blue, urlcolor=blue]{hyperref}
\usepackage{color}
\usepackage{psfrag,caption}

\usepackage{graphicx,epstopdf,verbatim}

\allowdisplaybreaks

\newtheorem{thm}{Theorem}[section]

 \newtheorem{prop}[thm]{Proposition}
 \theoremstyle{definition}
 \newtheorem{df}[thm]{Definition}
 \theoremstyle{remark}
 \newtheorem{rem}[thm]{Remark}
 
 \numberwithin{equation}{section}

\def\be#1 {\begin{equation} \label{#1}}
\newcommand{\ee}{\end{equation}}

\def\sqw{\hbox{\rlap{\leavevmode\raise.3ex\hbox{$\sqcap$}}$%
\sqcup$}}
\def\findem{\ifmmode\sqw\else{\ifhmode\unskip\fi\nobreak\hfil
\penalty50\hskip1em\null\nobreak\hfil\sqw
\parfillskip=0pt\finalhyphendemerits=0\endgraf}\fi}

\newcommand{\mb}{\medskip\noindent}
\newcommand{\gb}{\bigskip\noindent}
\newcommand{\R}{\mathbf R}

\newcommand{\N}{\mathbb N}
\newcommand{\Z}{\mathbb Z}

\newcommand{\s}{\mathcal S}

\newcommand{\cs}{\mathcal S}

\newcommand{\csp}{\mathcal S'}

\newcommand{\rn}{{\mathbf R}^n}
\newcommand{\rtwon}{{\mathbf R}^{2n}}
\newcommand{\bsexotic}{BS^0_{1,1;\pi/4}}

\begin{document}

\title
[Bilinear pseudodifferential operators]
{\bf  Sobolev space estimates for a class of bilinear 
pseudodifferential operators lacking symbolic calculus}

\author{Fr\'ed\'eric Bernicot}
\address{Fr\'ed\'eric Bernicot
\\
Laboratoire Paul Painlev\'e \\ CNRS - Universit\'e Lille 1 \\
F-59655 Villeneuve d'Ascq, France}
\email{frederic.bernicot@math.univ-lille1.fr}

\author{Rodolfo H. Torres}
\address{
Rodolfo H. Torres\\
Department of Mathematics\\
University of Kansas\\
Lawrence, KS 66045, USA}

\email{torres@math.ku.edu}

\thanks{Second author's research 
supported in part by the National Science Foundation under grant DMS 0800492}

\date{March 30, 2010}



\begin{abstract} The reappearance of a sometimes called exotic behavior for linear and multilinear pseudodifferential operators is investigated. The phenomenon is shown to be present in a recently introduced class of bilinear pseudodifferential operators which can be seen as more general variable coefficient counterparts of the bilinear Hilbert transform and other singular bilinear multipliers operators.  The unboundedness on product of Lebesgue spaces but the boundedness on spaces of smooth functions (which is the exotic behavior referred to) of such operators is obtained. In addition, by introducing a new way to approximate the product of two functions, estimates on a new paramultiplication are obtained.
\end{abstract}

\maketitle

\section{Introduction}
\subsection{An anomalous yet recurrent phenomenon}
This article is a continuation of recent work  devoted to the development  of a  
theory of bilinear and multilinear pesudodifferential operators which are the $x$-dependent counterparts of the singular multipliers modeled by the bilinear Hilbert transform.  In particular we will further study the class  of bilinear pseudodifferential operators $BS_{1,1; \, \pi/4}^0$ and show that it  has a sometimes called  {\it exotic} or {\it forbidden} behavior regarding boundedness on function spaces.

By a bilinear pseudodifferential operator we mean an operator, defined a priori on test functions, of the form
$$T_\sigma (f,g)(x)=\int_{\rtwon}\sigma (x, \xi, \eta )\widehat f(\xi )\widehat g(\eta )e^{ix\cdot(\xi + \eta )}
d\xi d\eta.$$
Two main types of $x$-dependent  classes of symbols have been studied in the literature.  One is the  Coifman-Meyer type $BS_{\rho, \delta}^m(\rn)$, $0\leq \delta \leq \rho \leq 1$, $m \in \R$, of symbols satisfying estimates of the form
\begin{equation}
\label{pseu}
|\partial _x^\alpha\partial _\xi ^\beta\partial _\eta ^\gamma \sigma (x, \xi ,\eta )|\leq C_{\alpha\beta\gamma}(1+|\xi |+|\eta |)
^{m+\delta |\alpha|-\rho (|\beta |+|\gamma |)},
\end{equation}
for all multi-indices $\alpha, \beta, \gamma$.

The other type  corresponds to classes denoted by $BS_{\rho, \delta; \, \theta}^m(\rn)$, $0\leq \delta \leq \rho \leq 1$, $m \in \R$, $-\pi/2<\theta \leq \pi/2$, and consisting of symbols satisfying
\begin{equation}
\label{tildepseu}
|\partial _x^\alpha\partial _\xi ^\beta\partial _\eta ^\gamma \sigma (x, \xi ,\eta )|\leq
C_{\alpha\beta\gamma; \theta}(1+|\eta -  \tan (\theta) \xi |)^{m+\delta |\alpha|-\rho (|\beta |+|\gamma |)}
\end{equation}
(where for $\theta=\pi/2$ the estimates are interpreted to decay in terms of $1+|\xi|$ only).
Both types of classes can be seen as bilinear analogs of the classical H\"ormander classes  $S_{\rho, \delta}^m(\rn)$ 
of linear pseudodifferential operators
$$T_\tau (f)(x)=\int_{\rn}\tau (x, \xi)\widehat f(\xi )e^{ix\cdot\xi}
d\xi,$$
with symbols satisfying 
\begin{equation}
\label{linearpseu}
|\partial _x^\alpha\partial _\xi ^\beta \tau (x, \xi )|\leq C_{\alpha\beta}(1+|\xi |)
^{m+\delta |\alpha|-\rho |\beta |}.
\end{equation}

As their name indicates, the first type of bilinear  classes were introduced by Coifman and Meyer at least  in the case 
$m=0$, $\rho=1$ and $\delta=0$,  \cite{come1}, \cite{come2}, \cite{come3}.  It is now well-understood that the operators in $BS^0_{1,0}$  are examples of certain singular integrals and  fit within the general multilinear Calder\'on-Zygmund theory developed by Grafakos and Torres \cite{gt1};  see also the works of Christ and Journ\'e \cite{CJ} and Kenig and Stein \cite{KS}. For other values of the parameters, the classes $BS_{\rho, \delta}^m$ were studied by B\'enyi \cite{ben}; B\'enyi and Torres \cite{beto1}, \cite{beto2}; B\'enyi et al \cite{bipseudo};  and more recently by B\'enyi et al \cite{bmnot}.  

The general classes $BS_{\rho, \delta; \, \theta}^m$ with $x$-dependent symbols were first introduced in \cite{bipseudo}.
A  connection to the bilinear Hilbert transform 
and the work of Lacey and Thiele \cite{lt1}, \cite{lt2} is given by the study in the $x$-independet case of singular multipliers  in one dimension satisfying 
$$
|\partial_\xi^\beta\partial_\eta^\gamma \sigma(\xi,\eta)|
 \leq C_{\beta\gamma} |\eta-\tan(\theta)\xi|^{-|\beta|-|\gamma|}.
 $$
This type of multipliers were investigated by Gilbert and
Nahmod \cite{gina1}, \cite{gina2}, \cite{gina3}; and Muscalu et al \cite{mtt}.   We  also recall that if for $\tau$ in $S^0_{1,0}(\R)$ we define 
\begin{equation}
\label{modulationinvariant}
\sigma(x,\xi,\eta)=\tau (x,\xi-\eta),
\end{equation} 
then $\sigma$ is  in $BS^0_{1,0; \pi/4}$. These operators have certain modulation invariance. Namely,
$$
T_\sigma(e^{iw\cdot}f, e^{iw\cdot}g)(x)= e^{i2w x}T_\sigma(f,g)(x)
$$
for all $w \in \R$. Such a $T_\sigma$ fits then within the more general  framework of modulation invariant bilinear singular integrals of B\'enyi et al \cite{t1bilineaire}.
Boundedness properties for symbols in the classes $BS_{1, 0; \, \theta}^0(\R)$,  not necessarily of the form \eqref{modulationinvariant}, were obtained by Bernicot \cite{pseudo1, pseudo}. We refer the reader to \cite{To3} for further motivation and references. 

In this article we want to discuss the reappearance of the exotic phenomenon  for the parameters $m=0$ and $\rho=\delta=1$. Namely, the unboundedness on $L^p$ spaces  of operators in  $BS_{1, 0; \, \theta}^0$, but  their boundedness on spaces of smooth functions.

In the linear case this phenomenon for $S_{1,1}^0$ is by now well-understood through the works of Stein \cite{stein}, Meyer\cite{meyer1},  Runst \cite{runst}, Bourdaud \cite{bourdaud},   H\"ormander \cite{hormander},  Torres \cite{To1}, among others. It is intimately related to the lack of calculus for the adjoints of operators in such class and,  ultimately, this behavior has been interpreted through the $T(1)$-Theorem of David and Journ\'e \cite{DJ}. The class $S_{1,1}^0$ is the largest class  of linear pseudodifferential operators with Calder\'on-Zygmund kernels but their exotic behavior on $L^p$ spaces is given by the fact that for $T$ in the class $S_{1,1}^0$,  the distribution $T^*(1)$ is in general not in $BMO$ (though $T(1)$ is).  Here $T^*$ is the formal transpose of $T$. Moreover, the boundedness of an
operator $T$ in $S_{1,1}^0$ on several other spaces of  function is related to the action (properly defined) of 
$T^*$ on polynomials;  see \cite{To2} and the relation to the work of H\"ormander \cite{hormander2} found in  \cite{To1}. By comparison, the smaller classes $S_{1,\delta}^0$ with $\delta <1$ are closed by transposition and hence the operators in such classes do satisfy the  hypotheses of the $T(1)$-Theorem and are bounded on $L^p$ for $1<p<\infty$.

Likewise, in the bilinear case, the class $BS_{1,1}^0$ is the largest class of pseudodifferential operators with bilinear Calder\'on-Zygmund kernels.  But gain, $T^{*1}$ and $T^{*2}$, the two formal transposes of an opearator $T$ in $BS_{1,1}^0$, may fail to satisfy the hypotheses of  the $T(1)$-Theorem for  bilinear Calder\'on-Zygmund operators in \cite{gt1}.  A symbolic calculus for the transposes hold in the smaller classes $BS_{1,\delta}^0$ with $\delta <1$, \cite{beto1}, \cite{bmnot}, rendering the boundedness of operators in $BS_{1,\delta}^0$. Though unbounded on product of $L^p$ spaces, the class  $BS_{1,1}^0$ is still bounded on product of Sobolev spaces \cite{beto1}.  For the Coifman-Meyer symbols  there is then a complete analogy with the linear situation.

For the newer more singular classes $BS_{1,0;\theta}^0$ a symbolic calculus for the transposes was 
shown to exist in \cite{bipseudo} and extended in \cite{pseudo}. Hence, the boundedness on product of $L^p$ spaces of operators in such classes and of the form 
\eqref{modulationinvariant} can be easily obtained from the new $T(1)$-Theorem for modulation invariant singular integrals in \cite{t1bilineaire}.  The class $BS_{1,0;\theta}^0$ also produced bounded operators on Sobolev spaces of positive smoothness as shown  in \cite{pseudo1}.  All these developments motivate us to look for exotic behavior in the larger classes 
$BS_{1,1;\theta}^0$.

\subsection{New results}
In this article, we show with an example that there exit  modulation invariant operators in the  class $BS_{1,1;\theta}^0$ which fail to be  bounded on product of $L^p$ spaces (Proposition \ref{prop:contreexemple}). This immediately implies that an arbitrary operator  $T$ in $BS_{1,1;\theta}^0$ may not have  both $T^{*1}(1,1)$ and $T^{*2}(1,1)$ in $BMO$, as defined in \cite{t1bilineaire}. It follows also that a symbolic calculus for the transposes in those classes is not possible. Nevertheless, as the reader may expect after the above introduction,  we shall show that the classes are bounded on product of Sobolev spaces. For simplicity in the presentation we will only consider the case $BS_{1,1;\pi/4}^0$. The corresponding results for other values of $\theta$ in  $(-\pi/2,\pi/2)\setminus \{-\pi/4\}$ (avoiding the degenerate directions) can be obtained in similar way.

In the case of modulation invariant operators, we obtained boundedness on product of Sobolev spaces with positive smoothness (Theorem \ref{thm1}). Surprisingly if we do not assume modulation invariance we can only obtain the corresponding result on Sobolev spaces of smoothness bigger than $1/2$ (Theorem \ref{thm2}). We do not know if the result is sharp,  but a better result  does not seem attainable with our techniques. Table~1 summarizes the known results and the new ones and puts  in evidence the parallel situation in several classes of pseudodifferential operators.

As a byproduct of our results, we also improve on some known estimates on paramultiplication by introducing a new way to approximate the pointwise product of two functions with 
errors better localized in the frequency plane (see Section \ref{paramulti} for precise statements).
\begin{table}
\begin{center}
\begin{tiny}
\begin{tabular}{|c|c|c|} \hline
 {\bf Class/symbol estimates} & {\bf Lebesgue spaces} & {\bf Sobolev spaces}
 \\ \hline

(linear) $S^0_{1,0}$ & $L^p \to L^p$ & $W^{s,p} \to W^{s,p}$  \\  
  $ |\partial^\beta_x\partial_{\xi}^\alpha \sigma (x, \xi)|\leq C_{\alpha \beta} (1+|\xi|)^{-|\alpha|} $
& $1<p<\infty$   & $1<p<\infty$, $s>0$ \\\hline
 
(linear)  $S^0_{1,1}$   & & $W^{s,p} \to W^{s,p}$  \\  
 $ |\partial^\beta_x\partial_{\xi}^\alpha \sigma (x, \xi)|\leq C_{\alpha \beta} (1+|\xi|)^{-|\alpha|} $ 
 &  unbounded  & $1<p<\infty$, $s>0$\\\hline

 (bilinear)  $BS^0_{1,0}$   & $L^p \times L^q \to L^t$  & $W^{s,p} \times W^{s,q} \to W^{s,t}$  \\  
 $|\partial^\beta_x\partial_{\xi,\eta}^\alpha \sigma (x, \xi,\eta)|\leq C_{\alpha \beta}  (1+|\xi|+|\eta|)^{-|\alpha|}$
  & $1<p, q<\infty$   & $1<p, q, t<\infty$, $s>0$
  \\
   &  $1/p + 1/q = 1/t $ & $1/p + 1/q = 1/t$
    \\\hline
 
  (bilinear)  $BS^0_{1,1}$   &  & $W^{s,p} \times W^{s,q} \to W^{s,t}$  \\  
 $|\partial^\beta_x\partial_{\xi,\eta}^\alpha \sigma (x, \xi,\eta)|\leq C_{\alpha \beta}  (1+|\xi|+|\eta|)^{|\beta|-|\alpha|}$
  & unbounded  & $1<p, q ,t<\infty$, $s>0$\\
   &  & $1/p + 1/q = 1/t $\\\hline

(bilinear)  $BS^0_{1,0; \pi/4}$   & $L^p \times L^q \to L^t$  & $W^{s,p} \times W^{s,q} \to W^{s,t}$  \\  
  $ |\partial^\beta_x\partial_{\xi,\eta}^\alpha \sigma (x, \xi, \eta)|\leq C_{\alpha \beta}  (1+|\xi-\eta|)^{-\alpha}$ 
  & $1<p, q<\infty$   &  $1<p, q, t<\infty$, $s>0$
  \\
   &  $1/p + 1/q = 1/t < 3/2$ & $1/p + 1/q = 1/t$
    \\\hline
  
  (bilinear)  $BS^0_{1,1; \pi/4}$   & & $W^{s,p} \times W^{s,q} \to W^{s,t}$  \\  
  $ |\partial^\beta_x\partial_{\xi,\eta}^\alpha \sigma (x, \xi - \eta)|\leq C_{\alpha \beta}  (1+|\xi-\eta|)^{\beta-\alpha}$ 
  & unbounded   & $1<p, q, t<\infty$, $s>0$\\
   &  & $1/p + 1/q = 1/t$
    \\\hline

  (bilinear)  $BS^0_{1,1; \pi/4}$   & & $W^{s,p} \times W^{s,q} \to W^{s,t}$  \\  
  $ |\partial^\beta_x\partial_{\xi,\eta}^\alpha \sigma (x, \xi , \eta)|\leq C_{\alpha \beta}  (1+|\xi-\eta|)^{\beta -\alpha}$ 
  & unbounded  & $1<p, q, t<\infty$, $s>1/2$\\
   &   & $1/p + 1/q = 1/t$
    \\\hline

 \end{tabular}
{\caption{Summary of the boundedness properties of pseudodifferential operators on Lebesgue and Sobolev spaces.}}
\end{tiny}
\end{center}
\end{table}%


\subsection{Further  definitions and notation}

We recall the maximal Hardy-Littlewood operator $M$ defined for a function $f\in L^1_{loc}(\R)$ by
$$M(f)(x) = \sup_{\genfrac{}{}{0pt}{}{B\ \textrm{ball}}{B\ni x}} \frac{1}{|B|} \int_B |f(y)| dy.$$
We write $M^2=M\circ M$ for the composition of the  maximal operator with itself.

For a function $f$ in the Schwartz space  $\cs$ of smooth  and rapidly decreasing functions,   we will  use the definition of the Fourier transform given by
$$\widehat f(\xi)=\int_{\R} f(x) e^{-ix\cdot \xi}\,dx.$$
With this definition, the inverse Fourier transform is given by
$f^{\vee} (\xi)=(2\pi)^{-1}\widehat f(-\xi)$.
Both the Fourier transform and its inverse  can be extended as usual to the dual space of tempered distributions $\csp$.  

For a bounded symbol $\sigma$, the bilinear operator 
$$ T_\sigma(f,g)(x)= \int e^{ix(\xi+\eta)} \widehat{f}(\xi) \widehat{g}(\eta) \sigma(x,\xi,\eta) \, d\xi d\eta$$
is well-defined and gives a bounded function for each pair of functions $f$, $g$ in $\cs $. Moreover for 
$\sigma$ in $BS^0_{1,1;\pi/4}$,  the operator $T_\sigma$ clearly maps $\cs \times \cs$ into $\csp$  continuously. This justifies many 
limiting arguments and computations that we will perform without further comment.

The formal transposes, $T^{*1}$ and  $T^{*2}$, of an operator $T:\cs \times \cs \to \csp$ are defined by 
$$\langle T^{*1}(h,g), f \rangle = \langle T(f,g), h\rangle = \langle T^{*2}(f,h), g \rangle,$$
where $\langle \cdot , \cdot \rangle$ is the usual pairing between distributions and test functions.

We will use the notation $\Psi_{2^{-k}}$ for the  $L^1$-normalized function $2^{k}\Psi(2^k \cdot)$ and consider  the Littlewood-Paley characterization of Sobolev spaces  $W^{s,p}$, $1<p<\infty$, $s\geq 0$. That is,   for
 a function $\Psi$ in $\cs$ with  spectrum contained in $\{\xi:\ 2^{-1}\leq |\xi|\leq 2\}$ and another function $\Phi$ also in $\cs$ and with spectrum included in $\{|\xi|\leq 1\}$, and such that 
\begin{equation}
\label{calderon}
 \widehat{\Phi}(\xi) + \sum_{k\geq 0} \widehat{\Psi}(2^{-k}\xi) =1 
 \end{equation}
for all $\xi$,  we have
\begin{equation}
\label{sobolev}
 \|f\|_{W^{s,p}} \approx \left\|\Phi \ast f\right\|_{L^p} + \left\|\left(\sum_{k\geq 0} 2^{2ks} \left|\Psi_{2^{-k}} \ast f \right|^2 \right)^{1/2}\right\|_{L^p}.
 \end{equation}
Here $\|\cdot\|_{L^p}$ denotes the usual norm of the Lebesgue space $L^p(\R)$.  For  $s=0$, the norm $\|\cdot\|_{W^{0,p}}$ is equivalent to $\|\cdot\|_{L^p}$. Also, by $BMO$ we mean as usual the classical John-Nirenberg space of functions of bounded mean oscillation. 

By homogeneity considerations, we will investigate boundedness properties of the form 
\begin{equation}
\label{bound}
T: W^{s,p} \times W^{s,q} \to W^{s,t},
\end{equation}
where the exponents  satisfy $1\leq p,q,t\leq \infty$ and the H\"older relation
\begin{equation}
\label{holder}
\frac{1}{p}+\frac{1}{q}=\frac{1}{t}.
\end{equation}
%
\section{Unboundedness on Lebesgue spaces}
We first show that for $s=0$ the bound \eqref{bound} may fail for $\bsexotic(\R)$.

\begin{prop} \label{prop:contreexemple} There exists a symbol $\tau\in S^0_{1,1}$ such that  the operator $T_\sigma$   with symbol $\sigma(x,\xi,\eta)= \tau(x,\xi-\eta)$ is in $\bsexotic$ and 
 is not bounded from $L^p \times L^q$ into $L^t$  for any exponents $p,q,t$ satisfying \eqref{holder}.
 \end{prop}

\begin{proof}
As in \cite{beto1}, we adapt to the bilinear situation a by now classical counterexample in the lineaar setting; see \cite{bourdaud}. Let $\psi$ be a function in $\cs$ satisfying $ \widehat{\psi}\geq 0$, $\widehat{\psi}(\xi)\neq 0$ only for  $5/7< |\xi|<5/3$, and  
$\widehat{\psi}(\xi)=1$ for  $5/6\leq |\xi|<4/3$. Consider the symbol 
$$
\tau(x,\xi)=\sum_{j\geq 4} e^{-i2^jx} \widehat{\psi}(2^{-j}\xi),
$$
which is easily seen to be in $S^0_{1,1}$. Select another function $\psi_1$ in $\cs$ verifying  
$ \textrm{supp }(\widehat{\psi_1}) \subset [0,1/3]$
and define
$$f=\sum_{j=4}^{m} a_j e^{i2^j x}\psi_1(x),$$
for arbitrarily coefficients $a_j$. For $\sigma(x,\xi,\eta)= \tau(x,\eta-\xi)$, we  have 
\begin{equation}
\label{oneterm}
 T_\sigma(f,\psi_1)(x)=\sum_{j,k \geq 4} a_k e^{-i2^jx}\int_{\R^2} e^{ix(\xi+\eta)} \widehat{\psi}(2^{-j}(\eta-\xi)) \widehat{\psi_1}(\xi-2^k)\widehat{\psi_1}(\eta) d\xi d\eta.
 \end{equation}
For each $k$, the integration at most takes place where $0\leq \eta \leq 1/3$ and $2^k\leq \xi \leq 2^{k}+1/3$,  which implies 
$$ -2^{k}-1/3 \leq \eta - \xi \leq 1/3-2^{k},$$ 
and then for each $j$,
\begin{equation}
\label{support}
 -2^{k-j}-2^{-j}/3 \leq 2^{-j}( \eta - \xi ) \leq 2^{-j}/3-2^{k-j}.
 \end{equation} 
Note that since $j,k \geq 4$, if $k>j$ we have 
$$2^{-j}/3-2^{k-j}< -5/3,$$
while if $k<j$ 
$$ -2^{k-j}-2^{-j}/3>-5/7.$$
It follows from \eqref{support} that the only non-zero term in \eqref{oneterm} is the one with $j=k$ and also
$$\widehat{\psi}(2^{-j}(\eta-\xi))=1$$
where the integrand is not zero.
We obtain
$$  T_\sigma(f,\psi_1)(x)= \sum_{j=4} ^m a_j e^{-i2^jx}e^{i2^jx}\psi_1^2(x) = \left(\sum_{j=4}^{m} a_j\right) \psi_1^2(x).$$
If we assume that the operator $T_\sigma$ is bounded from $L^p \times L^q$ into $L^t$, we could conclude then that
\be{contreex} \left|\sum_{j=4}^{m} a_j\right| \lesssim \|f\|_{L^p} \lesssim \left(\sum_{j=4}^{m} |a_j|^2 \right)^{1/2},\ee
where the last inequality  follows from  the Littlewood-Paley square function characterization of the $L^p$ norm of $f$ and the constants involved depend on $\psi_1$ but are independent of $m$. Since the $a_j$ are arbitrary \eqref{contreex} is not possible. 
\end{proof}

\section{Sobolev space estimates}
We will show that the class $\bsexotic$ produces bounded operators on product of Sobolev spaces. The situations in the modulation invariant and the general case are slightly different.

\subsection{The modulation invariant case}
We first consider the case of bilinear operators obtained from linear ones as in the previous section. That is,
the symbol $\sigma$ takes the form
$$ \sigma(x,\xi,\eta)=\tau(x,\xi-\eta),$$
where $\tau$ belongst to the linear class $S^{0}_{1,1}$. We have the following theorem.

\begin{thm} \label{thm1} Let $\tau$ be a linear symbol in $S^0_{1,1}$ and consider the bilinear operator $T_\sigma$, where $\sigma(x,\xi,\eta)=\tau(x,\xi-\eta)$. If $s>0$ and $1<p, q, t<\infty$ satisfy the H\"older relation \eqref{holder}, then $T_\sigma$  is bounded from $W^{s,p} \times W^{s,q}$ into $W^{s,t}$.
\end{thm}

\begin{proof}
We begin by recalling the  Coifman-Meyer reduction  for  symbols  in $S^0_{1,1}$(see e.g.  \cite{come3}, Chapter II,  Section 9), which is by now a standard technique in the subject. The symbol $\tau$ can be decomposed in a absolutely convergent sum of {\it reduced symbols} of the form
$$ \tau(x,\xi)=\sum_{j=0}^\infty m_j(2^jx) \widehat{\psi}(2^{-j}\xi),$$  
where $\psi$ is a smooth function whose Fourier transform is supported on $\{\xi:\ 2^{-1}\leq |\xi|\leq 2\}$ and $\{m_j\}_{j\geq 0}$ is a uniformely bounded collection of $C^r(\R)$ functions where $r$ can be taken arbitrarily large. Due to this reduction, we need only to study a symbol of the form
$$ \sigma(x,\xi,\eta)= \sum_{j\geq 0} m_j(2^j x)  \widehat{\psi} (2^{-j} (\xi-\eta)) :=\sum_{j\geq 0} \sigma_j(x,\xi,\eta).$$
We use the same notations of Bourdaud in \cite{bourdaud}. We  expand $m_j$ into an inhomogeneous Littlewood-Paley decomposition  using \eqref{calderon} so that
\be{decm} m_j = \sum_{k\geq 0} m_{j,k} \ee
with the spectrum of $m_{j,k}$ contained  in the dyadic annulus $\{\xi:\ 2^{k-1}\leq |\xi|\leq 2^{k+1}\}$ for $k\geq 1$, and in the ball $\{\xi,\ |\xi|\leq 2\}$ for $k=0$. Then we define for $h\geq j$ the function $n_{j,h}(x):= m_{j,h-j}(2^jx)$.  Due to the regularity of the function $m_j$, we have the following properties for $h\geq j+1$:
\be{prop1} \textrm{supp}\  \widehat{n_{j,h}} \subset \{\xi:\ 2^{h-1}\leq |\xi|\leq 2^{h+1}\}
\ee and
\be{prop2} \|n_{j,h}\|_{L^\infty} \leq C_r 2^{(j-h)r},\ee
where, we mention again, the number  $r$ can be chosen as large as we want. For $h=j$ we have
\be{prop11} \textrm{supp}\ \widehat{n_{j,j}} \subset \{\xi:\ |\xi|\leq 2^{j+1}\}\ee and
\be{prop21} \|n_{j,j}\|_{L^\infty} \leq C_r.\ee
Note also that,
\be{added}
m_j (2^j x)=  m_{j,k}(2^j x) +\sum_{h \geq j+1} m_{j,h-j}(2^j x)
= n_{j,j}(x) + \sum_{h \geq j+1} n_{j,h}(x).
\ee
Writing $T_j$ for the bilinear operator with symbol $\widehat{\psi}(2^{-j}(\xi-\eta))$, we get
\begin{align*}
T_\sigma(f,g)(x) & =\sum_{j\geq 0} m_j(2^jx) T_{j}(f,g)(x). 
 \end{align*}
To study the norm of $T_\sigma(f,g)$ in the Sobolev space $W^{s,t}$, and with the functions $\Psi$ and $\Phi$ as in (\ref{sobolev}), we need to estimate terms of the form $\Phi \ast T_\sigma(f,g)$ and, say for $k-2\geq 0$
$$\Psi_{2^{-k}}\ast T_\sigma(f,g):= \sum_{j\geq 0} \Psi_{2^{-k}} \ast [ m_j(2^j.)T_{j}(f,g)] = I_k(f,g)+ II_k(f,g),$$ 
where
$$ I_k(f,g):=\sum_{j= 0}^{k-2} \Psi_{2^{-k}} \ast [ m_j(2^j.)T_{j}(f,g)] $$
and
$$II_k(f,g):=\sum_{j\geq k-2} \Psi_{2^{-k}} \ast [ m_j(2^j.)T_{j}(f,g)].$$
We only treat $I_k$ and $II_k$. The estimate for the other terms can be achieved with the same arguments (they are actually easier). 
For  notational convenience, we identify $\Psi_{2^{-k}}$ with the convolution operator it defines (and similarly with other functions).

\mb {\bf Estimate for  $I$.} 
We further decompose $m_j(2^j.)$ and $T_{j}(f,g)$. Using (\ref{decm}), \eqref{added}, and (\ref{calderon}) we have
$$ m_j(2^jx) = \Phi_{2^{-k}}(m_j(2^j.))(x) + \sum_{l \geq k}  n_{j,l}(x).$$
We also decompose 
$$ T_j(f,g)(x) = {\Phi}_{2^{-k}}\left[T_j(f,g)\right](x) + \sum_{p \geq k}  {\Psi}_{2^{-p}} \left[T_j(f,g)\right](x).$$
We get
\begin{align} \label{ared}
I_k(f,g) & = \sum_{j=0}^{k-2} \Psi_{2^{-k}}\left[{\Phi}_{2^{-k}}(m_j(2^j.)) {\Phi}_{2^{-k}}(T_j(f,g)) \right] \\ 
 & \nonumber +\sum_{j=0}^{k-2} \sum_{l\geq k} \Psi_{2^{-k}}\left[ n_{j,l} {\Phi}_{2^{-k}}(T_j(f,g))\right] \\ 
 & \nonumber + \sum_{j=0}^{k-2}\sum_{p\geq k} \Psi_{2^{-k}}\left[ {\Phi}_{2^{-k}}(m_j(2^j.))  {\Psi}_{2^{-p}} \left[T_j(f,g)\right] \right]  \\
 & \nonumber + \sum_{j=0}^{k-2}\sum_{l,p\geq k} \Psi_{2^{-k}}\left[n_{j,l} {\Psi}_{2^{-p}} \left[T_j(f,g)\right] \right].
\end{align}
Using the notation $\widetilde{\phi}$ for a generic smooth function with bounded spectrum and 
$\widetilde{\psi}$ for a generic smooth function with a spectrum included into a corona around $0$, 
we claim that we can write $I_k$ as a sum of terms of  three different form:
$$ I_k(f,g)=\sum_{0\leq j \leq k-2} \Psi_{2^{-k}}(T_{\sigma_j}(f,g)) \approx (1)_k+(2)_k+(3)_k,$$
where
$$ (1)_k:= \sum_{j \leq k-2} \Psi_{2^{-k}}\left[n_{j,k} \widetilde{\phi}_{2^{-k}}\left[ T_j(f,g) \right] \right], $$
$$ (2)_k:= \sum_{j \leq k-2} \Psi_{2^{-k}} \left[\widetilde{\phi}_{2^{-k}}\left[m_j(2^j.) \right] \widetilde{\psi}_{2^{-k}}\left[ T_j(f,g) \right] \right],$$
$$ (3)_k:= \sum_{l \geq k} \sum_{j \leq k-2}  \Psi_{2^{-k}}\left[n_{j,l}\widetilde{\psi}_{2^{-l}}\left[ T_j(f,g) \right] \right].  $$ 
Let us explain this reduction. The  first sum in (\ref{ared})  can be written as  a finite linear combination of terms taking the form $(1)_k$ and $(2)_k$.  Indeed, consider one of the general terms $\Psi_{2^{-k}}\left[{\Phi}_{2^{-k}}(m_j(2^j.)) {\Phi}_{2^{-k}}(T_j(f,g)) \right]$ and  write $\xi$ (resp. $\eta$) for the frequency variable of $m_j(2^j.)$ (resp. $T_j(f,g)$). We have a non-vanishing contribution if
$$  |\eta|\leq 2^{k}, \quad |\xi|\leq 2^k \quad \textrm{and} \quad |\eta+\xi|\simeq 2^k,$$
where we have used that the spectrum of the product is included in Minkowski sum of spectrums.
Consequently, this is possible only if $|\xi|\simeq 2^k$, which corresponds to $(1)_k$ (recall that $n_{j,l}$ is frequentialy supported in $\left\{ |\xi|\approx 2^l\right\}$)), or $|\eta|\simeq 2^k$, which corresponds to $(2)_k$. 

Concerning the second sum in (\ref{ared}), it can also be reduced to the sum for $l\approx k$ (as the other terms vanish) and it is a finite sum of terms like $(1)_k$.  Similar  reasoning for the third term in (\ref{ared}) gives that it is controlled by $(2)_k$. Finally,  the general term in the fourth sum in (\ref{ared}) is non-zero if
$$ 2^p \pm 2^l \approx 2^k.$$
But, since the inner double sum  has  $l,p\geq k$, the general term is non-zero only for  $l\approx p$. We see then that the  double sum (over $l$ and $p$) reduces to one sum over only one parameter. It follows that the fourth sum in \eqref{ared}  is similar to $(3)_k$. \\
We now study each of the model sums $(1)_k$, $(2)_k$, $(3)_k$.

\mb $1)$ The sum with $(1)_k$. \\
We use the estimate (\ref{prop2}) for $n_{j,k}$  with $r>s$ and Young's inequality to obtain
\begin{align*}
\| 2^{ks} (1)_k \|_{l^2(k\in \N)} & \lesssim \left\| \sum_{j+2 \leq k} 2^{(j-k)r} 2^{ks} M\left(\widetilde{\phi}_{2^{-k}}\left[ T_j(f,g) \right] \right) \right\|_{l^2(k \in \N)} \\
 & \lesssim \left\| \sum_{j+2 \leq k} 2^{js}  2^{(j-k)(r-s)} M\left(\widetilde{\phi}_{2^{-k}}\left[ T_j(f,g) \right] \right) \right\|_{l^2(k\in \N)} \\
 & \lesssim \left\| 2^{js} M^2\left[ T_j(f,g) \right]   \right\|_{l^2(j\in \N)}.
\end{align*}
Therefore, 
\be{eq1} \left\|\| 2^{ks} (1)_k \|_{l^2(k\in \N)} \right\|_{L^t}  \lesssim \left\| \left\| 2^{js} M^2\left[ T_j(f,g) \right] \right\|_{l^2(j\in \N)} \right\|_{L^t} \ee
and from the Fefferman-Stein vector-valued inequality for the maximal operator $M$ (see \cite{FS}),
we deduce that
$$\left\|\| 2^{ks} (1)_k \|_{l^2(k\in \N)} \right\|_{L^t}  \lesssim \left\| \left\| 2^{js} T_j(f,g) \right\|_{l^2(j\in \N)} \right\|_{L^t}.$$
We can use now a linearization argument. By writing $r_j(\omega)$ for Rademacher functions ($\omega\in[0,1]$), we know that (see e.g. Appendix C in \cite{Gra}):
$$\left\|\| 2^{ks} (1)_k \|_{l^2(k\in\N)} \right\|_{L^t}  \lesssim \left\| \left\| \sum_{j} 2^{js}  r_j(\omega) T_j(f,g)  \right\|_{L^t(\omega \in [0,1])} \right\|_{L^t}.$$
By Fubini's Theorem, we have that
$$\left\|\| 2^{ks} (1)_k \|_{l^2(k\in\N)} \right\|_{L^t}  \lesssim \left\| \left\| \sum_{j} 2^{js}  r_j(\omega) T_j(f,g)  \right\|_{L^t} \right\|_{L^t(\omega\in [0,1])}.$$
Now for each $\omega\in [0,1]$, the operator $(f,g) \to \sum_{j} 2^{js}  r_j(\omega) T_j(f,g)$ is the bilinear operator associated to the symbol
$$ \sum_{j} 2^{js}  r_j(\omega) \widehat{\Psi} (2^{-j} (\xi-\eta)) \in BS^{s}_{1,0;\pi/4}.$$
From \cite{bipseudo} and \cite{pseudo} (since the symbol is $x$-independent)  these bilinear operators are bounded from $W^{s,p} \times W^{s,q}$ into $L^t$ (uniformly on $\omega\in[0,1]$) and  the  proof in this case is complete.

\gb $2)$ The sum with $(2)_k$. \\
This term is the most difficult to estimate. 
Using again  the boundedness of the functions $m_j$ in $C^r \hookrightarrow L^\infty$,  we can estimate
\begin{align}
\| 2^{ks} (2)_k \|_{l^2(k\in\N)} & \lesssim \left\| \sum_{j+2 \leq k} 2^{ks}M\left( \widetilde{\psi}_{2^{-k}}\left[ T_j(f,g) \right] \right)(x) \right\|_{l^2(k\in\N)} \label{pbm}.
 \end{align}
We observe that 
$$\widetilde{\psi}_{2^{-k}}\left[ T_j(f,g) \right] (x) = \int \widetilde{\psi}_{2^{-k}}(x-z)
 \int \widehat \Psi(2^{-j}(\xi-\eta)) \widehat f(\xi)  \widehat g(\eta)e^{i z(\xi+\eta)}\,d\xi d\eta\,dz
 $$
$$
= 
 \int \widehat {\widetilde\psi} (2^{-k}(\xi+\eta)) \widehat \Psi(2^{-j}(\xi-\eta))  \widehat f(\xi)  \widehat g(\eta)e^{i x(\xi+\eta)}\,d\xi d\eta.
$$
We must have  $|\xi+\eta| \approx 2^{k}$ and $|\xi-\eta| \approx 2^{j}$. But we only have terms with $2^j< 2^{k}/4$, so we deduce that $|\xi|\approx |\eta|\approx 2^{k}$. It follows that we can further localize  in the frequency plane adding a new function $\overline{\psi}$ (whose spectrum is included in a corona) such that
$$\widetilde{\psi}_{2^{-k}}\left[ T_j(f,g) \right] (x)
= \widetilde{\psi}_{2^{-k}}\left[ T_j(\overline{\psi}_{2^{-k}}f,\overline{\psi}_{2^{-k}}g) \right] (x)
$$
Going back to \eqref{pbm} we obtain by the Cauchy-Schwartz inequality (there are $k$ terms in the inner sum)
\begin{align*}
\| 2^{ks} (2)_k \|_{l^2(k\in\N)} & \lesssim \left\| 2^{ks}k^{1/2}  \left\| M\left(\widetilde{\psi}_{2^{-k}}\left[ T_j(\overline{\psi}_{2^{-k}}(f),\overline{\psi}_{2^{-k}}(g)) \right] \right) \right\|_{l^2(j\in\N)} \right\|_{l^2(k\in\N)}.
\end{align*}
We then  obtain similarly as in the previous case
$$\left\|\| 2^{ks} (2)_k \|_{l^2(k\in\N)} \right\|_{L^t}  \lesssim \left\| \left\| 2^{ks} k^{1/2}  \left\| M^2 \left[ T_j(\overline{\psi}_{2^{-k}}(f),\overline{\psi}_{2^{-k}}(g)) \right] \right\|_{l^2(j\in\N)} \right\|_{l^2(k\in\N)} \right\|_{L^t}$$
$$
 \lesssim \left\| \left\| 2^{ks} k^{1/2}  \left\| T_j(\overline{\psi}_{2^{-k}}(f),\overline{\psi}_{2^{-k}}(g)) \right\|_{l^2(j\in\N)} \right\|_{l^2(k\in\N)} \right\|_{L^t}.$$
We linearize in $j$ as before and using the fact that $k^{1/2} \lesssim 2^{ks}$ (as $s>0$),
$$\left\|\| 2^{ks} (2)_k \|_{l^2(k)} \right\|_{L^t}  \lesssim \left\| \left\| \left\| \sum_{j} r_{j}(\omega)  T_j(2^{ks}\overline{\psi}_{2^{-k}}(f),2^{ks}\overline{\psi}_{2^{-k}}(g)) \right\|_{L^1(\omega\in[0,1])} \right\|_{l^2(k\in\N)}  \right\|_{L^t}.$$
$$ \lesssim \left\| \left\|  \left\| \sum_{j} r_{j}(\omega)  T_j(2^{ks}\overline{\psi}_{2^{-k}}(f),2^{ks}\overline{\psi}_{2^{-k}}(g)) \right\|_{l^2(k\in\N)}  \right\|_{L^t} \right\|_{L^1(\omega\in[0,1])}.$$
For each $\omega\in[0,1]$, we can invoke a vector valued result for bilinear operators of Grafakos and Martell (\cite{GM}). More precisely, as explained in the first point, for each $\omega \in [0,1]$ the bilinear operator $(f,g) \to \sum_{j} r_{j}(\omega)  T_j(2^{ks}\overline{\psi}_{2^{-k}}(f),2^{ks}\overline{\psi}_{2^{-k}}(g)) $ is bounded from $L^p \times L^q$ to $L^t$ (since it is associated to a $x$-independent symbol). Then, Theorem 9.1 in \cite{GM} implies  that the operator admits an $l^2$-valued bilinear extension, which yields 
$$\left\|\| 2^{ks} (2)_k \|_{l^2(k\in\N)} \right\|_{L^t}  \lesssim \left\| \left\|\left\|2^{ks}\overline{\psi}_{2^{-k}}(f) \right\|_{l^2(k\in\N)} \right\|_{L^p} \left\| \left\| 2^{ks}\overline{\psi}_{2^{-k}}(g) \right\|_{l^2(k\in\N)}  \right\|_{L^q} \right\|_{L^1(\omega)},$$
with estimates uniformly in $\omega\in [0,1]$. This concludes the proof of the case  $(2)_k$.

\gb $3)$ The sum with $(3)_k$. \\
The analysis in this case  is entirely analogous as  the case $(1)_k$ and so we leave the details to the reader.

\gb
{\bf Estimate for  $II$.}
In this case, we decompose the term $ II_k(f,g)$ with quantities appearing as a linear combinaison of terms of the following form 
$$ (1)_k=  \sum_{j \geq k-2} \Psi_{2^{-k}}\left[n_{j,j} \widetilde{\phi}_{2^{-j}}\left[ T_j(f,g) \right]\right], $$
or
$$ (2)_k=  \sum_{j \geq k-2}  \sum_{l \geq j} \Psi_{2^{-j}}\left[ n_{j,l}(x)\widetilde{\psi}_{2^{-l}}\left[ T_j(f,g) \right] \right].  $$
Indeed with a similar reasoning as before and  since $j\geq k-2$, the general quantity in $II_k$ has a non-vanishing contribution only if the frequency variables of $m_j(2^j\cdot)$ or $T_j(f,g)$ are contained in $\{|\xi|\lesssim 2^j\}$ (which corresponds to $(1)_k$) or if the two frequency variables are contained in $\{|\xi|\simeq 2^l\}$ for some $l\geq j$ (which corresponds to $(2)_k$).

The study of $(2)_k$ is similar to the one of $(1)_k$ with the help of fast decays in $l$ (see (\ref{prop2})), so we only write the proof for $(1)_k$.
By the estimates on $n_{j,j}$, we have
$$\left\| \| 2^{ks}(1)_k \|_{l^2(k\in\N)} \right\|_{L^t} \lesssim \left\| \left\| \sum_{j\geq k-2} 2^{(k-j)s}  2^{js} M^2\left[ T_j(f,g) \right] \right\|_{l^2(k\in\N)} \right\|_{L^t}.$$
Using $s>0$ and Young's inequality for the $l^2$-norm on $k$, we get the following bound
$$ \left\| \| 2^{js} M^2\left[ T_j(f,g) \right] \|_{l^2(j\in\N)} \right\|_{L^t}.$$
We have already studied such quantities in the first case (see (\ref{eq1})) and prove the appropriate bounds. \\
The proof of the theorem in now complete.
\end{proof}

\begin{rem}
Since $\sigma(x,\xi,\eta)=\tau(x,\xi-\eta)$ is bounded, the function $T_\sigma(1,1)$  (rigorously defined in \cite{t1bilineaire}) is given by 
$$ T_\sigma(1,1)= \sigma(.,0,0) \in L^\infty \subset BMO.$$ 
If the transposes of  $T_\sigma$ are also given by symbols in the classes $BS^0_{1,1;\theta}$ or even by some bounded functions, then we can use the  bilinear $T(1)$-Theorem of \cite{t1bilineaire} (since $T_\sigma$ is modulation invariant) to conclude that  $T$ is bounded on the product of Lebesgue spaces.  The counterexample of the previous section shows that this is not always the case, so the classes $BS^0_{1,1;\theta}$ cannot be closed by transposition. As mentioned in the introduction the smaller classes $BS^0_{1,0;\theta}$  are.
\end{rem}

\subsection{The general case.}

In this subsection, we consider general symbols in the class $BS^0_{1,1;\pi/4}$. We obtain a slightly less general result than the one in the previous case.

\begin{thm} \label{thm2} If $\sigma\in BS_{1,1;\pi/4}$ and $s>1/2$, then the bilinear operator $T_\sigma$ is bounded from $W^{s,p} \times W^{s,q}$ into $W^{s,t}$ for all exponents $1<p,q,t<\infty$ satisfying the H\"older condition \eqref{holder}.
\end{thm}

\begin{proof} We want to adapt the proof of Theorem \ref{thm1}. We briefly indicate the extra difficulties faced. \\
{\bf Reduction to elementary symbols.} \\
We first reduce the problem to the study of elementary symbols taking the following form
\be{eq:form} \sigma(x,\xi,\eta)= \sum_{\genfrac{}{}{0pt}{}{j\geq 0}{l\in \Z}} m_{j,l}(2^j x)  \widehat{\Psi} (2^{-j} (\xi-\eta)) \widehat{\Psi} (l+2^{-j} (+\xi+\eta). \ee
Let us give a sketch of such reduction. By multiplying the symbol $\sigma$ with $\widehat{\Psi} (2^{-j} (\xi-\eta)) \widehat{\Psi} (l+2^{-j} (\xi+\eta)$, we  localize it  in the frequency to the following domain 
$$ \{(\xi,\eta), |\xi-\eta|\simeq 2^j \textrm{  and  } |\xi+\eta+l2^j| \simeq 2^j,$$
which can be compared to a ball of radius $2^j$. This compactly supported symbols  $\sigma_{j,l}$ satisfy
$$  |\partial^\alpha_x\partial_{\xi,\eta}^\beta \sigma_{j,l} (x, \xi ,\eta)|\leq C_{\alpha \beta}  2^{j(\alpha -\beta)}.$$
As usually, we decompose this symbol in Fourier series, obtaining
$$ \sigma_{j,l}(x,\xi,\eta) = \sum_{a,b\in\Z^2} \gamma_{a,b}(x) e^{i(a.\xi+b.\eta)} \widehat{\Psi} (2^{-j} (\xi-\eta)) \widehat{\Psi} (l+2^{-j} (+\xi+\eta).$$
The modulation term $e^{i(a.\xi+b.\eta)}$ does not play a role, as it corresponds to translation in physical space (which does not modify the Lebesgue norms), it remains  for us to check that the coefficients $\gamma_{a,b}$ are fast decreasing in $(a,b)$ and satisfies the desired smoothness in $x$. To do so, we remark that for $\alpha\in\N$ integrations by parts yields
\begin{align*}
 \left|\partial_x^\alpha \gamma_{a,b}(2^{-j} x)\right| & \lesssim 2^{-j\alpha} (2^j)^{-2} \left|\int\int e^{-i(a.\xi+b.\eta)} \partial_x^\alpha \sigma_{j,l} (2^{-j}x,\xi,\eta) d\xi d\eta \right| \\
& \lesssim 2^{-j\alpha} (2^j)^{-2} \left(1+|a|+|b|\right)^{-M} \\
& \hspace{1cm}\left|\int\int e^{-i(a.\xi+b.\eta)} \left(1+\partial_\xi^M+\partial_\eta^M\right) \partial_x^\alpha \sigma_{j,l} (2^{-j} x,\xi,\eta) d\xi d\eta \right| \\
& \lesssim \left(1+|a|+|b|\right)^{-M},
\end{align*}
where $M$ is an integer  that can be chosen as large as we wish. So we conclude that the functions $\gamma_{a,b}(2^{-j}\cdot)$ are uniformly bounded in $C^{r}$ (for $r$ arbitrarily large) with fast decays in $(a,b)$. This operation (expansion in Fourier series) allows us to reduce the study of $\sigma$ to reduced symbols taking the form (\ref{eq:form}).

\mb
{\bf Study of elementary symbols.} \\
We adapt the proof of Theorem \ref{thm1} and use the same notation.
We have to study the sum
\be{somme} 
\sum_{\genfrac{}{}{0pt}{}{j\geq 0}{l\in \Z}} m_{j,l}(2^j x) T_{j,l}(f,g),
 \ee
where $T_{j,l}$ is the bilinear operator associated to the $x$-independent symbol
$$
 \widehat{\Psi} (2^{-j} (\xi-\eta)) \widehat{\Psi} (l+2^{-j} (\xi+\eta)).
 $$
We can proceed as in the modulation invariant case and consider the different cases, eventually arriving to the point where we need to linearize with respect to the parameter $j$.  But now,  we also have to linearize according to the new parameter $l$. When we estimate the square  function of $T_{j,l}$, we  have to study $\Psi_{2^{-k}}(T_{j,l}(f,g))$ and we are interested only in the indices $j,l$ satisfying $|\xi+\eta|\approx 2^{k}$ with $|\xi-\eta|\approx 2^{j}$ and $|\xi+\eta+l2^j|\approx 2^{j}$.  However, due to the use of the Cauchy-Schwartz inequality in $l$, we will have an extra term bounded by $2^{(k-j)/2}$, which corresponds to the square root of the number of indices $l$ satisfying all these conditions. For the study of $(1)_k$ and $(3)_k$ there is no problem, since  $r$ can be chosen satisfying $r>s+1/2$. However, for the study of $(2)_k$ we will need $2^{k(s+1/2)}k^{1/2} \leq 2^{ks} 2^{ks}$ and so we need to assume that $s>1/2$.
\end{proof}

\begin{rem} It is interesting to note that without the modulation invariance, an extra exponent $1/2$ appears. We do not know if our result is optimal or not. Moreover, unlike the modulation invariance case, we  also do not know whether a general  operator $T_\sigma$ with symbol $\sigma \in BS_{1,1;\pi/4}$,  and whose two adjoints satisfy similar assumptions, is  bounded on product of Lebesgue spaces.  To address this question, it would be interesting to obtained (if possible) a  $T(1)$-Theorem as in \cite{t1bilineaire} but without assuming modulation invariance.
\end{rem}

\section{An improvement on paramultiplication. }\label{paramulti}

In this section, we will use $x$-independent symbols in $BS_{1,1;\pi/4}$ (and also in the smaller class $BS_{1,0;\pi/4}$) to describe a new paramultiplication operation. We will  obtain an improvement over the classical paramultiplication first studied by  Bony in \cite{bony} in the $L^2$ setting and extended by  Meyer in \cite{meyer1,meyer2} to $L^p$ norms. The classical paraproducts and their properties hold for multidimensional variables, however our improvement works (at least at this moment) only in the  one dimensional case. \\
We start with the classical definition.
\begin{df} Let $f$ and $b$ be two smooth functions and let $\Phi$ and $\Psi$ be as in (\ref{calderon}) and  (\ref{sobolev}). We assume that for all $\eta \in {\rm supp}\,\widehat \Phi$ and $\xi\in {\rm supp}\,\widehat \Psi$ we have
$$ |\eta| \leq \frac{1}{2} |\xi|.$$
Then the paramultiplication by $b$ is defined as follows
$$ \Pi_b(f) := \sum_{k\in \Z} \Phi_{2^k}(f) \Psi_{2^k}(b).$$
\end{df}
The operator $(b,f)\to \Pi_b(f)$ can essentially be thought as a bilinear multiplier whose symbol is a  smooth decomposition of the characteristic function  of the  cone  in Figure~\ref{fig1}.

\begin{figure}
 \psfragscanon
\psfrag{xi1}[l]{$\eta$}
\psfrag{xi2}[l]{ $\xi$}
\psfrag{diag}[l]{$\xi=\eta$}
\psfrag{diag2}[l]{ $\xi=-\eta$}
\includegraphics[width=0.4\textwidth]{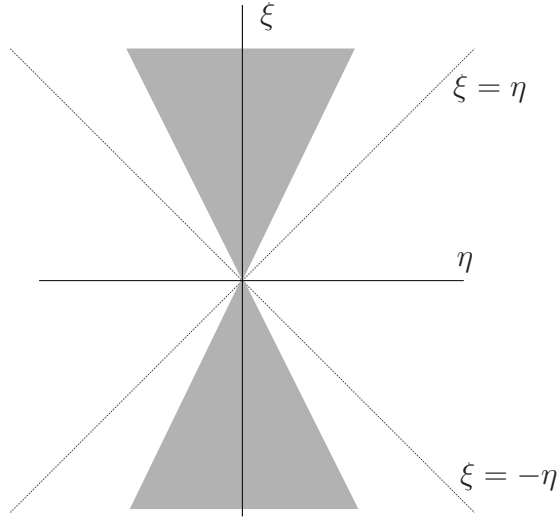}
\caption{Support of the bilinear symbol associated to the paraproduct $\Pi$.}
\label{fig1}
\end{figure}

\mb The following two propositions are well-known properties for paraproducts (see e.g. Theorems 2.1 and 2.5 in \cite{bony} for the originel results involving $L^2$-Sobolev spaces and \cite{meyer1,meyer2} for extension to other Sobolev spaces): 

\begin{prop} \label{prop:conti} For all $s>0$ and $p\in (1,\infty)$ the linear operator 
$\Pi_b$ is bounded on the Sobolev space $W^{s,p}$, satisfies
$$ \left\| \Pi_b \right\|_{W^{s,p} \to W^{s,p}} \lesssim \left\| b \right\|_{L^\infty},$$
and  the operation can be extended  to an $L^\infty$ function $b$.
\end{prop}
The paramultiplication approximates the pointwise multiplication is the following sence.

\begin{prop} \label{prop:error} Let $1< t < \infty$ and $s>1/t$. For $f\in W^{s,t}$ and $g\in W^{s,t}$,  we have
$$ \left\| fg - \Pi_f(g)-\Pi_g(f) \right\|_{W^{2s-1/t,t}} \lesssim \left\| f \right\|_{W^{s,t}} \left\| g \right\|_{W^{s,t}}.$$
\end{prop}
The exponent of regularity $2s-\frac{1}{t}$  is bigger than $s$ for $ts>1$. 
This gain is very important. The result is essentialy due to the fact that, in the frequency space, the error term has only a contribution from $f$ and $g$ when  
$$ \left\{ |\xi| \approx |\eta| \right\},$$
i.e.  in a cone along the two main diagonals.

\mb Using the new bilinear operators (whose singularities are localized on a line in the frequency plane), we can define a new paramultiplication operation $\widetilde{\Pi}$ such that the error term will be  concentrated in the frequency plane  exactly in a strip (of fixed width) around the two diagonals. In this way, we will be able to get a better gain for the exponent of regularity. 

\begin{df} \label{def:para} Let $\Theta$ be a smooth function on $\R$ such that its Fourier transform $\widehat{\Theta}$ satisfies
 $$ \omega\geq 2 \Longrightarrow \widehat{\Theta}(\omega)=1 \qquad -\infty<\omega\leq 1 \Longrightarrow \widehat{\Theta}(\omega)=0.$$
Then we define for $b,f\in \s(\R)$, the {\it improved paramultiplication} by $b$ (written $\widetilde{\Pi}_b(f)$) by
\begin{equation}
\label{para}
 \widetilde{\Pi}_b(f)(x)=
\int_{\R^2} e^{ix(\xi+\eta)} \widehat{b}(\xi) \widehat{f}(\eta) \left[\widehat{\Theta}(\xi-\eta) \widehat{\Theta}(\xi+\eta) + \widehat{\Theta}(\eta-\xi) \widehat{\Theta}(-\xi-\eta) \right] d\xi d\eta. 
\end{equation}
\end{df}

\mb The new bilinear multiplier $(b,f)\to \widetilde{\Pi}_b(f)$ is associated to a bilinear symbol, corresponding to a smooth version of the characteristic function of the region in Figure~\ref{fig2}. We remark  that this new region approximates  the domain $\{(\xi,\eta), |\xi|\geq |\eta|\}$
better than the region in Figure~\ref{fig1}.

\begin{figure}
 \psfragscanon
\psfrag{xi1}[l]{$\eta$}
\psfrag{xi2}[l]{ $\xi$}
\psfrag{diag}[l]{$\xi=\eta$}
\psfrag{diag2}[l]{ $\xi=-\eta$}
\includegraphics[width=0.4\textwidth]{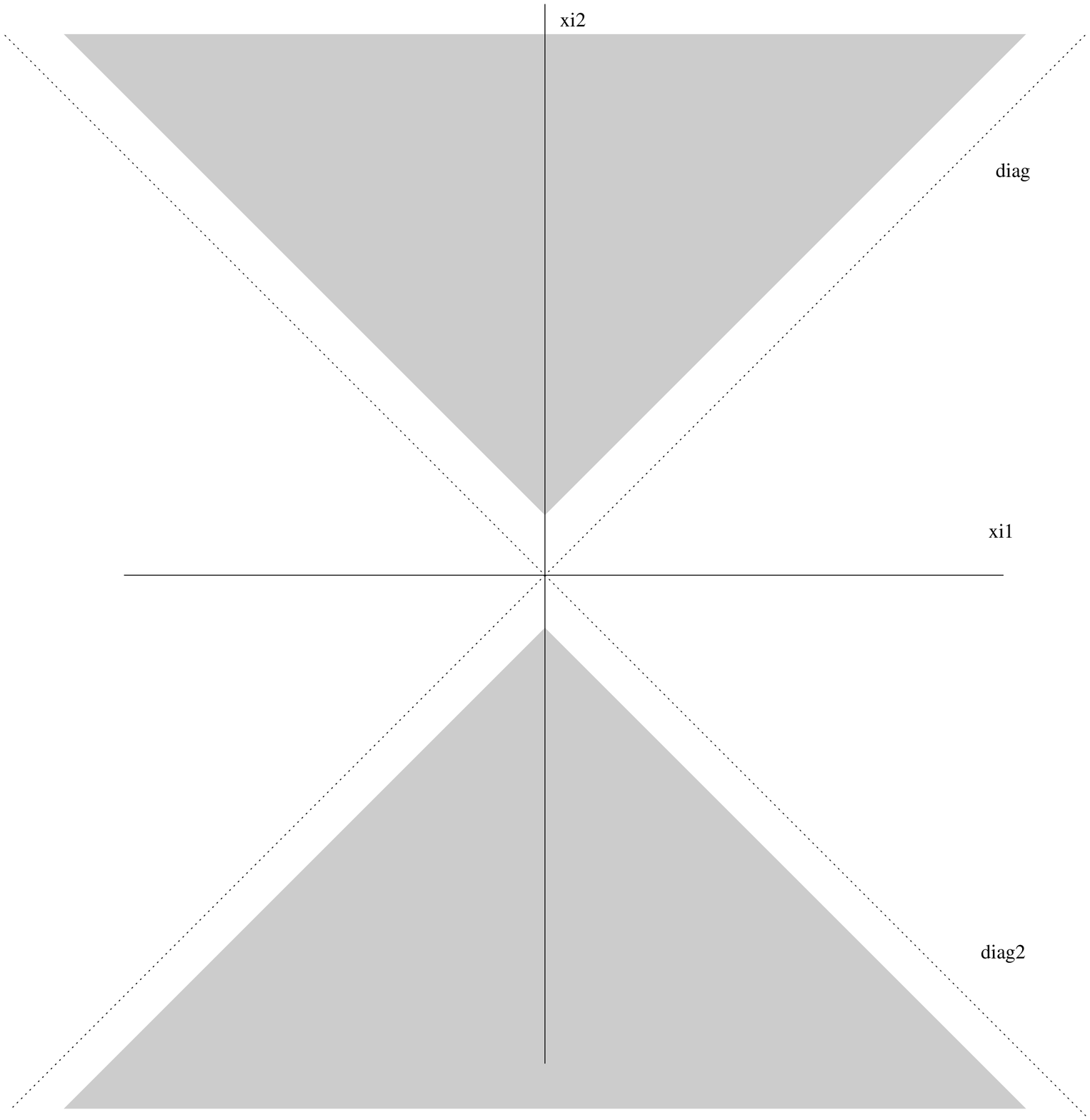}
\caption{Support of the bilinear symbol associated to the new paraproduct $\widetilde{\Pi}$.}
\label{fig2}
\end{figure}

\mb This new operation satisfies a similar property to the one in Proposition \ref{prop:conti}.

\begin{prop} \label{prop:continuite} Let $s\geq 0$  and  let $1<p,q,t <\infty$ be exponents satisfying (\ref{holder}). For every $\epsilon>0$ and $b\in W^{\epsilon,p}(\R)$, the improved paramultiplication by $b$ is well-defined  and produce a bounded operation from $W^{s,q}$ to $W^{s,t}$. In fact, there exists a constant $C=C(s,\epsilon,p,q,t)$ such that for all functions  $f\in W^{s,q}$,
 $$ \left\| \widetilde{\Pi}_b(f) \right\|_{W^{s,t}} \leq C \|b\|_{W^{\epsilon,p}} \|f\|_{W^{s,q}}.$$
Moreover if $s=0$, the exponent $\epsilon=0$ is allowed.
\end{prop}

\begin{proof} The new paramultiplication is given by two terms, which can be studied by identical arguments. We only deal with the first term but for simplicity in the notation we still write
$$ \widetilde{\Pi}_f(b)(x)= \int_{\R^2} e^{ix(\xi+\eta)} \widehat{b}(\xi) \widehat{f}(\eta) \widehat{\Theta}(\xi-\eta) \widehat{\Theta}(\xi+\eta) d\xi d\eta. 
$$ 
We note that this function $\widetilde{\Pi}_b(f)$ corresponds to the operator $T_\sigma(b,f)$ associated to the bilinear  symbol
$$ \sigma(\xi,\eta)= \widehat{\Theta}(\xi-\eta) \widehat{\Theta}(\xi+\eta).$$
We need to show that  $T_\sigma$ is continuous from $W^{\epsilon,p}\times W^{s,q}$ to $W^{s,r}$. 

\mb {\bf The case $s=0$.} \\
We compute the Fourier transform of $T_\sigma(b,f)$,
\begin{align*}
\widehat{T_\sigma(b,f)}(\omega) & = \int_{\xi+\eta=\omega} \widehat{b}(\xi) \widehat{f}(\eta) \widehat{\Theta}(\xi-\eta) \widehat{\Theta}(\eta+\xi) d\xi d\eta \\
 & = \widehat{\Theta}(\omega) \int_{\xi+\eta=\omega} \widehat{b}(\xi) \widehat{f}(\eta) \widehat{\Theta}(\xi-\eta)  d\xi d\eta \\
 & = \widehat{\Theta}(\omega) \widehat{T_\tau(b,f)}(\omega),
\end{align*}
where $\tau$ is given by $\tau(\xi,\eta)= \widehat{\Theta}(\xi-\eta)$. So in fact we can write $T_\sigma(b,f)$ as the convolution product between $\Theta$ and $T_\tau(b,f)$. Since the function $\Theta$ in Definition \ref{def:para} is smooth,  the convolution operation by $\Theta$ is bounded on $L^t$. We obtain also
$$ 
\left\| T_\sigma(b,f) \right\|_{L^t} \lesssim \left\|T_\tau(b,f) \right\|_{L^t}.
$$
Now the bilinear operator $T_\tau$ is associated to the symbol $\tau$ which satisfies the H\"ormander multiplier conditions related to the frequency line $\{\xi=\eta\}$. That is,
$$ \left| \partial_\xi^\alpha \partial_\eta^\beta \tau(\xi,\eta) \right| \lesssim \left|\xi-\eta\right|^{-\alpha-\beta}
$$
for all $\alpha$ and $\beta$.
It follows from the work of Gilbert and Nahmod \cite{gina1} that this bilinear operator maps $L^p \times L^q$ to 
$L^t$ and we obtain  the desired result 
$$ \left\| T_\sigma(b,f) \right\|_{L^t} \lesssim \|b\|_{L^p} \|f\|_{L^q}.$$
Note that for the case $s=0$ no regularity on $b$ is really needed.

\mb {\bf The case $s>0$.} \\
Let  $\Phi$ and $\Psi$ be as in (\ref{calderon}) and (\ref{sobolev}).
We study first $\Phi \ast T_\sigma(f,g)$. We have
$$ \widehat{\left[\Phi \ast T_\sigma(b,f)\right]}(\omega)=  \widehat{\Phi}(\omega)\widehat{\Theta}(\omega) \widehat{T_\tau(b,f)}(\omega).$$
The spectral condition over $\Phi$ and $\Theta$ imply that $\omega \approx 1$. So for $\xi$ and $\eta$ (the frequency variables of $b$ and $f$) satisfying $ \xi - \eta \geq 1$ and $\xi+\eta=\omega\approx 1$,  we deduce that either $\eta$ is bounded or $-\xi \approx \eta >>1$. Therefore,  we can find a smooth function $\zeta$ and an other one $\widetilde{\psi}$ (whose spectrum is contained in a corona around $0$) such that
$$ \Phi \ast T_\sigma (b,f) = \Phi \ast T_\sigma (b,\zeta \ast f) + \sum_{l\geq 0} \Phi \ast T_\sigma (\widetilde{\psi}_{2^{-l}} \ast b,\widetilde{\psi}_{2^{-l}}\ast f).$$
Using $0<\epsilon$, we get by the Cauchy-Schwartz inequality
\begin{align}
 \left|\Phi \ast T_\sigma (b,f) \right| & \leq  \left|\Phi \ast T_\sigma (b,\zeta \ast f) \right| + \left(\sum_{l\geq 0} 2^{2\epsilon l} \left|M\left[T_\sigma (\widetilde{\psi}_{2^{-l}} \ast b,\widetilde{\psi}_{2^{-l}}\ast f) \right] \right|^2 \right)^{1/2}.
\label{eq10}
\end{align}
By the same reasoning for an integer $k\geq 1$, if $\xi$ and $\eta$ satisfy $\eta\geq \xi+1$ and $1<\xi+\eta =\omega \approx 2^k$, we deduce that either $\eta \approx 2^k$ or $-\xi \approx \eta >> 2^{k}$. So we can find a smooth function $\widetilde{\psi}$ (for convenience we keep the same notation), whose spectrum is included in a corona around $0$ such that for all integer $k$ large enough
$$ \Psi_{2^{-k}} \ast T_\sigma(b,f) = \Psi_{2^{-k}} \ast T_\sigma (b,\widetilde{\psi}_{2^{-k}} \ast f) + \sum_{l\geq k} \Psi_{2^{-k}} \ast T_\sigma (\widetilde{\psi}_{2^{-l}} \ast b,\widetilde{\psi}_{2^{-l}} \ast f).$$
Using the same $\epsilon$, we get by Minkowski's and Cauchy-Schwartz'  inequalities 
\begin{align}
\left(\sum_{k} 2^{2ks} \left|\Psi_{2^{-k}} \ast T_\sigma(b,f)\right|^2 \right)^{1/2} & \lesssim \left(\sum_{k} 2^{2ks} M\left[ T_\sigma (b,\widetilde{\psi}_{2^{-k}} \ast f)\right] ^2 \right)^{1/2} \nonumber \\
 & \hspace{0cm} + \sum_{l\geq 0} \left( \sum_{k\leq l} 2^{2ks} \left|\Psi_{2^{-k}}\ast T_\sigma (\widetilde{\psi}_{2^{-l}} \ast b,\widetilde{\psi}_{2^{-l}} \ast f) \right|^2 \right)^{1/2} \nonumber \\
 & \lesssim \left(\sum_{k} 2^{2ks} M\left[ T_\sigma (b,\widetilde{\psi}_{2^{-k}} \ast f)\right] ^2 \right)^{1/2} \nonumber \\
 &  \hspace{0cm} + \sum_{l\geq 0} 2^{ls} M\left[T_\sigma (\widetilde{\psi}_{2^{-l}} \ast b,\widetilde{\psi}_{2^{-l}} \ast f) \right] \nonumber \\
 &  \lesssim \left(\sum_{k} 2^{2ks} M\left[ T_\sigma (b,\widetilde{\psi}_{2^{-k}} \ast f)\right] ^2 \right)^{1/2} \nonumber \\
 &   \hspace{0cm} + \left(\sum_{l\geq 0} 2^{2l(s+\epsilon)} \left| M\left[T_\sigma (\widetilde{\psi}_{2^{-l}} \ast b,\widetilde{\psi}_{2^{-l}} \ast f) \right] \right|^2 \right)^{1/2}.   \label{eq20}
\end{align}
From (\ref{eq10}) and (\ref{eq20}),  using the $L^q-L^{t}$ boundedness of $T_\sigma(b,.)$ (proved in the first case), the vector-valued Fefferman-Stein inequality,  and its bilinear version (Theorem~9.1 of \cite{GM} already mentioned), we obtain the desired result:
\begin{align*} 
 \| T_\sigma(b,f)\|_{W^{s,t}} & \lesssim \left\| \left| \Phi \ast T_\sigma(b,f) \right| + \left( \sum_{k \geq 0} 2^{2sk} \left| \Psi_{2^{-k}} \ast T_\sigma(b,f) \right|^2 \right)^{1/2} \right\|_{L^t} \\  
 & \lesssim \|b\|_{L^p} \left\| \left|\zeta \ast f \right| + \left( \sum_{k \geq 0} 2^{2sk} \left| \widetilde{\psi}_{2^{-k}} \ast f \right|^2 \right)^{1/2} \right\|_{L^q} \\
 &  \hspace{1cm} + \left\| \left( \sum_{l \geq 0} 2^{2l\epsilon} \left| \widetilde{\psi}_{2^{-l}} \ast b \right|^2 \right)^{1/2}\right\|_{L^p}  \left\| \left( \sum_{k \geq 0} 2^{2sk} \left| \widetilde{\psi}_{2^{-k}} \ast f \right|^2 \right)^{1/2}\right\|_{L^q}
 \\
 & \lesssim \|b\|_{W^{\epsilon,p}} \|f\|_{W^{s,q}}.
\end{align*}
\end{proof}

\begin{rem} We note that our new bilinear operation needs an extra regularity assumption $b\in W^{\epsilon,p}$ to keep the regularity of the function $f$ (the case $s>0$). This is due to the fact that the high frequencies of $b$ play a role in the high frequency of $\widetilde{\Pi}_b(f)$ (which is natural) but in the low frequencies of $\widetilde{\Pi}_b(f)$ too. This last phenomenom does not appear in the classical paramultiplication operation. This point can be observed in the Figures \ref{fig1} and \ref{fig2}. Let $\omega$ be the frequency variable of the paraproduct. For small $\omega$, say $\omega \simeq 2$, the contributions of $b$ and $f$  correspond to the intersection of the cone  in Figures 1 and 2 and the line $\{\omega=\xi+\eta\}$. In the first case (Figure \ref{fig1}) this intersection is bounded set, whereas in the second case (Figure \ref{fig2}) it is not bounded and contains also high frequencies of $b$.
\end{rem}

\mb We now obtain an improvement on Proposition \ref{prop:error}. 

\begin{prop} \label{prop:produit} Let $t\in(1,\infty)$ and $s\geq 1/t$. If $f\in W^{s,t}$ and $g\in W^{s,t}$, then
 $$ \left\|fg - \widetilde{\Pi}_f(g) - \widetilde{\Pi}_g(f) \right\|_{W^{2s,t}} \lesssim \|f\|_{W^{s,t}} \|g\|_{W^{s,t}}.$$
\end{prop}
 
\begin{rem} As already mentioned, in the classical paramultiplication calculus, the regularity result is true for $s\geq 1/t$ and the gain is only $s-1/t$. 
\end{rem}
 
\begin{proof} 
Let us denote by  $D$ the difference operator 
$$ D(f,g):= fg - \widetilde{\Pi}_f(g) - \widetilde{\Pi}_g(f).$$
It corresponds to the bilinear operator associated to the symbol $\tau$ given by
\begin{align*}
 \tau(\xi,\eta)  := & 1 - \widehat{\Theta}(\eta-\xi) \widehat{\Theta}(\eta+\xi) - \widehat{\Theta}(-\eta+\xi) \widehat{\Theta}(-\eta-\xi) \\
 & -\widehat{\Theta}(\xi-\xi) \widehat{\Theta}(\eta+\xi) - \widehat{\Theta}(\eta-\xi) \widehat{\Theta}(-\eta-\xi).
\end{align*}
This symbol is supported in the complement of the cone drawn in Figure \ref{fig2} and the symetrical one. Consequently, it is supported in two strips (around the two diagonals)
$$ \textrm{supp} (\tau) \subset \left\{(\xi,\eta),\ |\xi-\eta|\leq 3\right\} \cup \left\{(\xi,\eta),\ |\xi+\eta|\leq 3\right\}.$$
We can then reproduce a similar reasonning as used for Proposition \ref{prop:continuite}. 
The symbol $\tau$ can be decomposed in two parts $\tau_1,\tau_2$; the first one supported in $\left\{(\xi,\eta),\ |\xi+\eta|\leq 3\right\}$ and the second one supported in $\left\{(\xi,\eta),\ |\xi-\eta|\leq 3\right\}$. 

The bilinear multiplier associated to $\tau_1$ has only low frequencies,  hence
$$ \|T_{\tau_1}(f,g)\|_{W^{2s,t}} \lesssim \|T_{\tau_1}(f,g)\|_{L^{t}}.$$
Using Proposition \ref{prop:continuite} with exponents $t, p,q\in (1,\infty)$ satisfying (\ref{holder}), it follows that
$$ \|T_{\tau_1}(f,g)\|_{W^{2s,t}} \lesssim \|f\|_{L^p} \|g\|_{L^q} \lesssim \|f\|_{W^{s,t}} \|g\|_{W^{s,t}},$$
where we have used the  Sobolev embeding $W^{s,t} \subset L^p$ since $s\geq 1/t>1/t-1/p$ (and similarly with $q$). 

Concerning the second part $\tau_2$, it is easy to check that, on its support , $1+|\xi+\eta|$, $1+|\xi|$ and $1+|\eta|$ are comparable and in addition
\be{compa} \max \{1+|\xi+\eta|, 1+|\xi|,\ 1+|\eta| \} - \min \{1+|\xi+\eta|, 1+|\xi|,\ 1+|\eta| \} \lesssim 1.\ee
We claim that $T_{\tau_2}$ is bounded from $L^t\times L^t$ into $L^t$. 
Indeed,  the symbol $\tau_2$ is supported around the diagonal $\xi=\eta$ and it takes the  form
$$ \tau_2(\xi,\eta) = m(\xi-\eta),$$
for a smooth function $m$ supported on $[-3,3]$. It follows  that
\be{aze} T_{\tau_2}(f,g)(x)=\int \widehat{m}(y) f(x-y)g(x+y) dy. \ee
Since $m\in \s(\R)$ we have, in particular, that $\widehat{m}\in L^1\cap L^\infty$, and using Minkowski's inequality
we easily deduce that 
$T_{\tau_2}$ is bounded from $L^\infty \times L^\infty$ to $L^\infty$ and from $L^1 \times L^1$ to $L^1$. 
By (complex) bilinear interpolation, we conclude that $T$ is bounded from $L^t \times L^t$ to $L^t$, for 
$1<t<\infty$.

It remains to estimate $T_{\tau_2}$ in the Sobolev space. We let the reader verify that, as in similar previously done computations  
 (and using (\ref{compa})), $T_{\tau_2}$ can be decomposed as 
\be{eq:dec} T_{\tau_2}(f,g) = \sum_{k\geq 0} \Psi_{2^{-k}} T_{\tau_2}\left( \Psi_{2^{-k}}^1 f,\Psi_{2^{-k}}^2 g\right), \ee
for some smooth frequency truncations $\Psi,\Psi^1,\Psi^2$. It follows that
\begin{align*}
\left\| \left(\sum_{k\geq 0} 2^{k4s}  \left| \Psi_{2^{-k}} T_{\tau_2}  \left( \Psi_{2^{-k}}^1 f,\Psi_{2^{-k}}^2 g\right)\right|^2\right)^{1/2} \right\|_{L^t} 
 & \\
& \hspace{-5cm} \lesssim \left\|\left(\sum_{k\geq 0} 2^{k4s} \left| T_{\tau_2}\left( \Psi_{2^{-k}}^1 f, \Psi_{2^{-k}}^2 g\right)\right| ^2\right)^{1/2} \right\|_{L^t} \\
 & \hspace{-5cm} \lesssim \left\|\left(\sum_{k\geq 0} 2^{ks} \left| \Psi_{2^{-k}}^1 f \right|^2\right)^{1/2}\right\|_{L^t} \left\|\left(\sum_{k\geq 0} 2^{k2s} \left| \Psi_{2^{-k}}^2 g \right| ^2\right)^{1/2} \right\|_{L^t} \\
& \hspace{-5cm} \lesssim \| f\|_{W^{s,t}} \| g\|_{W^{s,t}},
\end{align*}
where we have used  the $L^t$  boundedness of the operator $T_{\tau_2}$ and its $l^2$-vector valued extension (given again by Theorem 9.1 of \cite{GM}).
\end{proof}

\begin{rem} 
The previous proof relies on the boundedness from $L^t \times L^t$ to $L^t$ of $T_{\tau_2}$. This property does not hold in  the classical paraproduct situation.

We have given a proof by interpolation, where the specific form of $\tau_2$ plays an important role. We would like to describe now a direct proof of the boundedness for the simpler case $t=2$. The arguments are based on the geometric fact that the symbol $\tau_2$ is supported on a strip around the diagonal with bounded width.

We can use in the $L^2$ case  a partition of frequencies 
given by $\Delta_k$ a smooth  truncation on the interval $[k-4,k+4]$:
$$ \widehat{\Delta_k(f)}(\xi)=\chi(\xi-k)\widehat{f}(\xi),$$
where $\chi$ is a smooth function, supported on $[-4,4]$ and equal to $1$ on $[-3,3]$. Then, by Plancherel's equality, we have
$$ \| T_{\tau_2}(f,g)\|_{L^2}\lesssim \left( \sum_{k\in \Z} \| \Delta_k(T_{\tau_2}(f,g)\|_{L^2}^2\right)^{1/2}.$$
By (\ref{compa}), it follows that with other similar truncation operators $\Delta^1$ and $\Delta^2$,
\begin{align*} 
\| T_{\tau_2}(f,g)\|_{L^2} & \lesssim \left( \sum_{k\in \Z} \| \Delta_k(T_{\tau_2}(\Delta_k^1(f),\Delta_k^2(g))) \|_{L^2}^2\right)^{1/2} \\
 & \lesssim \left( \sum_{k\in \Z} \| {\bf 1}_{|\xi-k|\leq 4} \int \left| \widehat{\Delta_k^1(f)}(\eta) \widehat{\Delta_k^2(g)}(\xi-\eta) \right|d\eta \|_{L^2}^2\right)^{1/2} \\
& \lesssim \left( \sum_{k\in \Z} \left\| \widehat{\Delta_k^1(f)}\right\|_{L^2}^2  \left| \widehat{\Delta_k^2(g)}\right\|_{L^2}^2\right)^{1/2},
\end{align*}
where we have used that each interval $[k-4,k+4]$ has a bounded length. Since the collection $([k-4,k+4])_{k\in\Z}$ is a bounded covering, we can conclude the boundedness of $T_{\tau_2}$ from $L^2\times L^2$ into $L^2$.  (Note that the same argument does not apply in $L^p$.)
\end{rem}

\begin{rem} 
Our new definition of paramultiplication is based on bilinear operators associated to $x$-independent symbols of the class $BS_{1,0;\pi/4}$. We could use the Sobolev boundedness (proved in the first sections of the current paper) in order to define other kind of paramultiplications with an $x$-dependent symbol but we will not carry here such analysis any further.
\end{rem}

\end{document}